\documentclass[12pt,leqno]{amsart}
\usepackage{amsfonts,amsthm,amsmath,mathtools,amssymb,comment,enumitem}

\usepackage{ascmac}

\usepackage{braket}
\usepackage{stmaryrd}
\usepackage{bm}

\usepackage{tikz}
\tikzstyle{v} = [circle, draw, inner sep=1pt, minimum size=3pt, fill=black]

\theoremstyle{plain}
\newtheorem{thm}{Theorem}[section]
\newtheorem{prop}[thm]{Proposition}
\newtheorem{lem}[thm]{Lemma}
\newtheorem{cor}[thm]{Corollary}

\theoremstyle{definition}
\newtheorem{df}[thm]{Definition}
\newtheorem{rem}[thm]{Remark}
\newtheorem{ex}[thm]{Example}
\newtheorem{prob}[thm]{Problem}

\newcommand{\NN}{\mathbb{N}}

\newcommand{\Hom}{\mathrm{Hom}}
\newcommand{\wHom}{\mathrm{Hom}^{\mathrm{w}}}
\newcommand{\Homw}{\widehat{\mathrm{Hom}}}
\newcommand{\Sur}{\widehat{\mathrm{Sur}}}

\begin{document}

\title[Universal graph series]
{Universal graph series and vertex-weighted version of chromatic symmetric function}

\author[Yosuke Sato]{YOSUKE SATO}
\address
{
School of Fundamental Science and Engineering, 
Waseda University, Tokyo 169-8555, Japan
}
\email{y.sato.248@ruri.waseda.jp}  

\date{\today}

\begin{abstract}
We focus on two specific generalizations of the chromatic symmetric function: one involving universal graphs and the other concerning vertex-weighted graphs. In this paper, we introduce a unified generalization that incorporates both approaches and demonstrate that the resulting new invariants inherit characteristics from each, particularly the properties of complete invariants. Additionally, we construct complete invariants for directed acyclic graphs (DAGs) and partially ordered sets (posets). As a corollary, these invariants can distinguish hyperplane arrangements that are distinguishable by their intersection posets.
\end{abstract}


\maketitle

\section{Introduction}

Let $G=(V(G),E(G))$ be a simple graph with vertex set $V(G)$ and edge set $E(G)$. A function $\kappa:V(G)\to\mathbb{N}$ is called a proper coloring if, for any edges $\{u,v\}\in E(G)$, $\kappa(u)\ne\kappa(v)$.

\begin{df}[\cite{Sta95}]\label{df1.1}
   Let $x=\{x_1,x_2,\cdots\}$ be countably indeterminates. The chromatic symmetric function of $G$ is defined as
  \[
  X(G)=X(G,x)=\sum_{\kappa}\prod_{v\in V(G)}x_{\kappa(v)},
  \]
  where $\kappa$ runs all proper colorings.
\end{df}

\begin{df}
  Let $G$ and $H$ be simple graphs. A map $\varphi:V(G)\to V(H)$ is called a homomorphism if, for any edge $\{u,v\}\in E(G)$, $\{\varphi(u),\varphi(v)\}$ are edges in $E(H)$. $\Hom(G,H)$ denotes a set of homomorphisms from $G$ to $H$.
\end{df}

$\Hom(G,K_\mathbb{N})$ is a set consisting of all proper colorings of $G$ with colors in $\mathbb{N}$, where $K_{\mathbb{N}}$ denotes the infinite complete graph on $\mathbb{N}$. Hence the sum in Definition \ref{df1.1} is taken over $\Hom(G,K_\mathbb{N})$:
\[
X(G)=\sum_{\varphi\in\Hom(G,K_{\mathbb{N}})}\prod_{v\in V(G)}x_{\varphi(v)}.
\]

The chromatic symmetric function $X(G)$ has the following expansion, known as the power sum expansion.

\begin{thm}
\[
X(G)=\sum_{S\subseteq E}(-1)^{|S|}p_{\lambda_S}
\]
where, for each subset $S\subseteq E$, $\lambda_S$ denotes the partition of $|V(G)|$ given by the vertex counts in each connected component. For a partition $\lambda=(\lambda_1,\cdots,\lambda_l)$, we define
\[
p_{\lambda}=p_{\lambda_1}\cdots p_{\lambda_l}.
\]
\end{thm}

\section{Background}

\subsection{Universal graph series and chromatic symmetric function}

The concept of the universal graph was introduced by Rado in \cite{Rado} and generalized to the universal graph series in \cite{MMNST}.

\begin{df}[\cite{MMNST},\cite{Rado}]
  (1) Let $G$ be an infinite graph. $G$ is called universal if it contains every finite graph as an induced subgraph.

  (2) Let $H=\{H_n\}_{n\in N}\ (N\subset\mathbb{N})$ be a family of graphs. $H$ is called a universal graph series if, for any simple graph $G$, there exists $n\in N$ such that $H_n$ includes $G$ as an induced subgraph. Furthermore, the universal graph series $\{H_n\}_{n\in\mathbb{N}}$ is referred to as induced universal if, for any pair $m,n\in\mathbb{N}$, $H_m$ is an induced subgraph of $H_n$.
\end{df}

The family of Kneser graphs $\{K_{\mathbb{N},k}\}_{k\in\mathbb{N}}$ is induced universal \cite{HPW}. The Kneser graph $K_{\mathbb{N},k}$ has $\binom{\mathbb{N}}{k}$ as its vertex set and for any vertices $u$ and $v$ in $\binom{\mathbb{N}}{k}$, $u$ and $v$ are adjacent if $u \cap v = \emptyset$.

\begin{df}
  Let $G$ be a simple graph and $H$ be a graph. The \textit{$H$-chromatic function} is defined as follows:
  \[
  X_H(G)=X_H(G,x)=\sum_{\varphi\in\Hom(G,H)}\prod_{v\in V(G)}x_{\varphi(v)},
  \]
  where $x=\{x_u\}_{u\in V(H)}$ and $x_u$ are indeterminates. 
\end{df}

The complete invariants of finite graphs were constructed using universal graphs.

\begin{thm}[\cite{MMNST}]\label{completeness}
  (1) Let $H$ be a universal graph. Then
  \[
  X_H(\bullet)
  \]
  is a complete invariant for finite graphs.

  (2) Let $H=\{H_n\}_{n\in\mathbb{N}}$ be a universal graph series. Then
  \[
  \{X_{H_n}(\bullet)\}_{n\in\mathbb{N}}
  \]
  is a complete invariant for finite graphs.
\end{thm}

Specifically, $X_{K_{\mathbb{N},k}}(G)$ is called the $k$-\textit{multiple chromatic function}. When $k=1$, this reduces to a chromatic symmetric function. $k$-multiple chromatic function has a generalized power sum expansion:

\begin{thm}[\cite{MMNST}, Theorem 1.9]\label{k-multiple_powersum}
\begin{align*}
X_{K_{\mathbb{N},k}}(G) = \sum_{S \subseteq E(G)}(-1)^{|S|}\sum_{\lambda \in \mathcal{A}_{G_S}^{(k)}}p_{\lambda}.
\end{align*}
$\mathcal{A}_{G_S}^{(k)}$ and $p_\lambda$ are defined in Section \ref{sec:Kneser}.
\end{thm}

\subsection{Extension of the chromatic symmetric function to the vertex-weighted graphs.
}

  A vertex-weighted graph $(G,w)$ is composed of a graph $G$ and a weight function $w:V(G)\to\mathbb{N}$.

\begin{df}[\cite{CS}]
  Let $(G,w)$ be a vertex-weighted graph and let $x=\{x_1,x_2,\cdots\}$ be countably indeterminates. Crew and Spirkl extended the chromatic symmetric function to vertex-weighted graphs as follows:
  \[
  X(G,w)=X((G,w),x)=\sum_{\varphi\in\Hom(G,K_\mathbb{N})}\prod_{v\in V(G)}x_{\varphi(v)}^{w(v)}.
  \]
  If all weights are $1$, $X(G,w)=X(G)$.
\end{df}

Let $G$ be a graph, and let $e = \{v_1, v_2\} \in E(G)$ be any edge. The graph $G \backslash e$ represents the graph obtained by removing the edge $e$ from $G$. The graph $G / e$ represents the graph obtained by removing the edge $e$ and identifying the two vertices $v_1$ and $v_2$ as a single vertex. The modified weight function $w/e$ on the contracted graph $G/e$ is defined as follows: if $e$ is a loop, then $w/e$ remains the same as $w$; otherwise, it is given by
\[
(w/e)(u) =
\begin{cases}
    w(u), & \text{if } u \ne v_1, v_2, \\
    w(v_1) + w(v_2), & \text{if } u \text{ is the contracted vertex in\ }G/e.
\end{cases}
\]
If a multiedge arises due to contraction, it may be replaced by a single edge, since this operation does not affect the chromatic symmetric function.

$X(G,w)$ satisfies the deletion-contraction relation.

\begin{prop}[\cite{CS}, Lemma 2]
    \[
    X(G,w)=X(G\backslash e,w)-X(G/e,w/e).
    \]
\end{prop}

$X(G,w)$ also has a power sum expansion:

\begin{thm}[\cite{CS}, Lemma 3]
    \[
    X(G,w)=\sum_{S\subseteq E(G)}(-1)^{|S|}p_{\lambda(G,w,S)},
    \]
    where $\lambda(G,w,S)$ denotes the partition consisting of total weights of the connected components of $(G_{S},w)$, with $G_{S}=(V(G),S)$.
\end{thm}

\section{Extending $X_H(G)$ to Vertex-Weighted Graphs}

\begin{df}
  Let $(G,w)$ be a vertex-weighted graph and $H$ be a graph. We extend the $H$-chromatic function to vertex-weighted graphs as follows:
  \[
  X_H(G,w)=X_H((G,w),x)=\sum_{\varphi\in\Hom(G,H)}\prod_{v\in V(G)}x_{\varphi(v)}^{w(v)},
  \]
  where $x_u$ $(u\in V(H))$ are indeterminates. We adopt the notation
  \[
  x_{\varphi}(G,w):=x_{\varphi}((G,w),H,x):=\prod_{v\in V(G)}x_{\varphi(v)}^{w(v)},
  \]
  where $\varphi\in\Hom(G,H)$.
\end{df}

We have complete invariants for vertex-weighted graphs in a manner similar to Theorem \ref{completeness}.

\begin{thm}\label{main_thm}
  (1) Let $H$ be a universal graph. Then
  \[
  X_H(\bullet)
  \]
  is a complete invariant for finite vertex-weighted graphs.

  (2) Let $H=\{H_n\}_{n\in\mathbb{N}}$ be a universal graph series. Then
  \[
  \{X_{H_n}(\bullet)\}_{n\in\mathbb{N}}
  \]
  is a complete invariant for finite vertex-weighted graphs.
\end{thm}

To prove this theorem, we introduce weight-homomorphisms for vertex-weighted graphs.

\begin{df}\label{wt:hom}
  Let $(G_1,w_1)$ and $(G_2,w_2)$ be simple vertex-weighted graphs. We call a map $\psi:V(G_1)\to V(G_2)$ a weight-homomorphism if $\psi$ satisfies the following conditions.
  
  $(1)$ $\psi$ is a homomorphism.
  
  $(2)$ For any $v\in V(G_2)$, $\displaystyle\sum_{u\in\psi^{-1}(v)}w_1(u)\le w_2(v).$

  We write $\Homw\big((G_1,w_1),(G_2,w_2)\big)$ for a set of weight-homomorphisms from $(G_1,w_1)$ to $(G_2,w_2)$. 
\end{df}

\begin{df}\label{wt:cong}
  Let $(G_1,w_1)$ and $(G_2,w_2)$ be vertex-weighted graphs. We say $(G_1,w_1)$ and $(G_2,w_2)$ are isomorphic if they satisfy the following two conditions:
  
  $(1)$ $G_1$ and $G_2$ are isomorphic. Namely, there exists a bijection $\varphi:V(G_1)\to V(G_2)$ such that for any two vertices $u,v\in V(G_1)$, $\{u,v\}\in E(G_1)$ if and only if $\{\varphi(u),\varphi(v)\}\in E(G_2).$ Such a map is called an isomorphism.
  
  $(2)$ For any isomorphism $\varphi:V(G_1)\to V(G_2)$, $w_2\circ\varphi=w_1$.
\vskip\baselineskip
$(G_1,w_1)\cong(G_2,w_2)$ denotes that $(G_1,w_1)$ and $(G_2,w_2)$ are isomorphic.
\end{df}

We extend the necessary and sufficient conditions for graph isomorphism (\cite{GR}, {\rm p.128 Exercise 11}) to vertex-weighted graphs.

\begin{lem}\label{lem:wt-hom}
  Let $(G_1,w_1)$ and $(G_2,w_2)$ be vertex-weighted graphs. If $\big|\Homw\big((G_1,w_1),(F,w)\big)\big|=\big|\Homw\big((G_2,w_2),(F,w)\big)\big|$ for any finite vertex-weighted graph $(F,w)$, then $(G_1,w_1)\cong(G_2,w_2)$.
\end{lem}

\begin{proof}
  Let $(G_1,w_1)$ and $(G_2,w_2)$ be vertex-weighted graphs. Assume that $\big|\Homw\big((G_1,w_1),(F,w)\big)\big|=\big|\Homw\big((G_2,w_2),(F,w)\big)\big|$. First, we show that
  \[
  \big|\Sur\big((G_1,w_1),(F,w)\big)\big|=\big|\Sur\big((G_2,w_2),(F,w)\big)\big|
  \]
  by induction on $|V(F)|$, where $\Sur\big((X,w_X),(Y,w_Y)\big)$ denotes a set of surjective weight-homomorphism from $(X,w_X)$ to $(Y,w_Y)$. When $|V(F)|=1$, the equality clearly holds because any mapping to a single-vertex graph is surjective. Assume that the equality holds for $|V(F)|=n$. We then show that it holds for $|V(F)|=n+1$.
  \begin{align*}
     \big|\Sur\big((G_1,w_1),(F,w)\big)\big|&=|\Homw\big((G_1,w_1),(F,w)\big)\big|-\sum_{W\subsetneq F}|\Sur\big((G_1,w_1),(W,w)\big)\big|\\
     &=|\Homw\big((G_2,w_2),(F,w)\big)\big|-\sum_{W\subsetneq F}|\Sur\big((G_2,w_2),(W,w)\big)\big|\\
     &=\big|\Sur\big((G_2,w_2),(F,w)\big)\big|
  \end{align*}
  Here, we use the lemma's hypothesis for the first term and the induction hypothesis for the second term. Therefore, the equality is established.

  Now, set $(F,w)=(G_1,w_1)$, which gives us the following inequality:
\[
\big|\Sur\big((G_2,w_2),(G_1,w_1)\big)\big| = \big|\Sur\big((G_1,w_1),(G_1,w_1)\big)\big| > 0.
\]
From this, we deduce that $|V(G_1)| \leq |V(G_2)|$.  Additionally, for any $\psi_2\in\Sur\big((G_2,w_2),(G_2,w_2)\big)$, 
  \[
 \sum_{v'\in\psi_2^{-1}(v)}w_2(v')\le w_1(v).\tag{$*$}
  \]

Next, by setting $(F,w)=(G_2,w_2)$, we similarly obtain $|V(G_2)| \leq |V(G_1)|$ and for any $\psi_1\in\Sur\big((G_1,w_1),(G_2,w_2)\big)$,
  \[
  \sum_{v\in\psi_1^{-1}(v')}w_1(v)\le w_2(v').\tag{$**$}
  \]

Given that both inequalities $|V(G_1)| \leq |V(G_2)|$ and $|V(G_2)| \leq |V(G_1)|$ hold, we conclude that $|V(G_1)| = |V(G_2)|$. Thus, the number of vertices in both graphs must be equal. Hence, $\psi_1\in\Sur\big((G_1,w_1),(G_2,w_2)\big)$ must be injective, thus an isomorphism. Therefore, condition (1) of Definition \ref{wt:cong} is satisfied.

Let $\varphi:V(G_1)\to V(G_2)$ be an arbitrary isomorphism between $G_1$ and $G_2$. We will show that $w_2\circ\varphi=w_1$ by induction on $n=|V(G_1)|=|V(G_2)|$. When $n=1$, by $(*)$ and $(**)$, $(w_2\circ\varphi)(u)=w_1(u)\ (u\in V(G_1))$. Assume that $w_2\circ\varphi=w_1$ holds for $n$. Consider $|V(G_1)|=|V(G_2)|=n+1$. Let $v_m$ be one of the vertices in $V(G_1)$ with the smallest weight. If $w_2(\varphi(v_m))>w_1(v_m)$, it contradicts $(*)$. Reversing the direction of the inequality similarly contradicts $(**)$, so $w_2(\varphi(v_m))=w_1(v_m)$. Define the restricted weight functions $w_1\backslash{v_m}:=w_1|_{G_1\backslash v_m}$ and $w_2\backslash{\varphi(v_m)}:=w_2|_{G_2\backslash\varphi(v_m)}$. Since $|(V(G_1\backslash v_m))|=|(V(G_2\backslash\varphi(v_m))|=n$, by the induction hypothesis, $(w_2\backslash\varphi(v_m))\circ\varphi=w_1\backslash v_m$. Combining this with $(w_2\circ\varphi)(v_m)=w_1(v_m)$, we obtain $w_2\circ\varphi=w_1$. Therefore, the lemma is proved.
\end{proof}

Now we give the proof of Theorem \ref{main_thm}.

\begin{proof}[Proof of Theorem \ref{main_thm}]
  We provide a proof of (1) only, as (2) follows similarly. Suppose that $X_H(G_1,w_1,x)=X_H(G_2,w_2,x)$, where $x = {x_u}{u \in V(H)}$ denotes a set of indeterminates. Let $(F,w)$ be a finite vertex-weighted graph. Since $H$ is a universal graph, $F$ can be embedded as an induced subgraph of $H$. 

Define
\[
m(F,w):=m((F,w),x):=\prod_{v\in V(F)}x_v^{w(v)}.
\]
For each \(x_v\), define a map \(\gamma_F\) by
\[
\gamma_F(x_v) := \begin{cases}
    1, & \text{If}\ v \in V(F), \\
    0, & \text{If}\ v \notin V(F).
\end{cases}
\]
Next, for any monomial $M(x)$, define the map $\Gamma_{(F,w)}$ by
\[
\Gamma_{(F,w)}(M(x)):=\begin{cases}
M(\gamma_F(x)), & \text{If } m(F,w)/M(x)\in\mathbb{Z}[x],\\
0, & \text{otherwise}.
\end{cases}
\]
Then, we have
\[
\big|\Homw\big((G_i, w_i), (F, w)\big)\big| = \sum_{\varphi\in\Hom(G_i,H)}\Gamma_{(F,w)}\big(x_{\varphi}(G_i,w_i)\big).\quad(i=1,2)
\]
This holds because if $m(F, w)/x_{\varphi}(G_i,w_i)\in\mathbb{Z}[x]$, $\varphi$ is a weight-homomorphism from $(G,w)$ to the subgraph of $(F,w)$ and $\Gamma_{(F,w)}(x_{\varphi}(G_i,w_i))=1$.

Therefore, from the first assumption,

\begin{align*}
\big|\Homw\big((G_1,w_1),(F,w)\big)\big| &= \sum_{\varphi\in\Hom(G_1,H)} \Gamma_{(F,w)}\big(x_\varphi(G_1,w_1)\big) \\
&= \sum_{\varphi\in\Hom(G_2,H)} \Gamma_{(F,w)}\big(x_\varphi(G_2,w_2)\big)\\
&= \big|\Homw\big((G_2,w_2),(F,w)\big)\big|.
\end{align*}
Thus, by the lemma \ref{lem:wt-hom}, we conclude that \((G_1, w_1) \cong (G_2, w_2)\).
\end{proof}

\section{Power sum expansion of $X_{K_{\mathbb{N},k}}(G,w)$}\label{sec:Kneser}

To obtain the power sum expansion, we introduce weak homomorphisms and $\mathcal{P}^{(k)}$.

\begin{df}[\cite{MMNST}]
Let $G$ and $H$ be graphs. A mapping $\varphi: V(G) \to V(H)$ is called a weak homomorphism if, for every edge $\{u, v\} \in E(G)$, either $\{\varphi(u), \varphi(v)\} \in E(H)$ or $\varphi(u) = \varphi(v)$ holds. We denote a set of all weak homomorphisms from $G$ to $H$ by $\wHom(G, H)$.
\end{df}

We extend the definition of $W_H(G)$ from \cite{MMNST} to vertex-weighted graphs as follows:
\[
W_H(G, w) = W_H((G, w), x) = \sum_{\varphi \in \wHom(G, H)} \prod_{v \in V(G)} x_{\varphi(v)}^{w(v)}.
\]

\begin{prop}[\cite{MMNST}]\label{prop:weak_expansion}
For a subset $S \subseteq E(G)$, let $G_S$ denote the spanning subgraph of $G$. Then, the following holds:
\[
X_H(G, w) = \sum_{S \subseteq E(G)} (-1)^{|S|} W_{\overline{H}}(G_S, w),
\]
where $\overline{H}$ represents the complement of $H$.
\end{prop}
\begin{proof}
This result follows by replacing $G$ with $(G,w)$ in the proof of Proposition 3.1 in \cite{MMNST}.
\begin{align*}
    \sum_{S \subseteq E(G)} (-1)^{|S|} W_{\overline{H}}(G_S, w) &= \sum_{S \subseteq E(G)} (-1)^{|S|} \sum_{\varphi \in \wHom(G_S, \overline{H})} \prod_{v \in V(G)} x_{\varphi(v)}^{w(v)}\\
    &= \sum_{\varphi : V(G) \to V(H) } \sum_{S \subseteq E_\varphi} (-1)^{|S|} \prod_{v \in V(G)} x_{\varphi(v)}^{w(v)},
\end{align*}
where
\begin{align*}
    E_\varphi &= \{\{u,v\} \in V(G) \mid \varphi(u)=\varphi(v)\ \text{or}\ \{\varphi(u),\varphi(v)\} \in E(\overline{H})\}\\
    &= \{\{u,v\} \in E(G) \mid \{\varphi(u),\varphi(v)\} \notin E(H)\}.
\end{align*}
Since
\[
  \sum_{S\subseteq E_\varphi}(-1)^{|S|}=\begin{cases}
    1,&\text{if}\ E_\varphi=\emptyset,\\
    0,&\text{otherwise}.
  \end{cases}
\]
It follows that
\[
  \sum_{\varphi : V(G) \to V(H)} \sum_{S \subseteq E_\varphi} (-1)^{|S|} \prod_{v \in V(G)} x_{\varphi(v)}^{w(v)} = X_H(G, w).
\]
\end{proof}

In \cite{MMNST}, Miezaki et al. introduced an equivalence relation $\sim$ on the set of $n$-element multisets composed of elements from $\binom{\mathbb{N}}{k}$. For two multisets $\{I_1, \cdots, I_n\}$ and $\{J_1, \cdots, J_n\}$, we define $\{I_1, \cdots, I_n\} \sim \{J_1, \cdots, J_n\}$ if there exists a permutation $\sigma \in S_{\mathbb{N}}$ such that, as multisets,
\[
\{I_1, \cdots, I_n\} = \{\sigma(J_1), \cdots, \sigma(J_n)\}.
\]
\begin{df}[\cite{MMNST}]
We denote by $\mathcal{P}_n^{(k)}$ the set of equivalence classes. Furthermore, we define $\mathcal{P}^{(k)} := \bigsqcup_{n=1}^\infty \mathcal{P}_n^{(k)}$.
\end{df}

Let $\lambda \in \mathcal{P}_n^{(k)}$ and assume $\lambda \sim \{I_1, \cdots, I_n\}$. Then, $\lambda$ is viewed as a $k$-regular hyper-multigraph, where the vertex set is $V_\lambda = I_1 \cup \cdots \cup I_n$ and the edge set is $E_\lambda = \{I_1, \cdots, I_n\}$. In particular, $\mathcal{P}^{(2)}$ corresponds to the set of all multigraphs without loops or isolated vertices.

For $\lambda \in \mathcal{P}_n^{(k)}$, the monomial $k$-fold symmetric function $m_\lambda = m_\lambda^{(k)}$ is defined as follows:
\[
m_\lambda := \sum_{\{I_1, \cdots, I_n\} \in \lambda} x_{I_1} \cdots x_{I_n}.
\]
If the hypergraph $\lambda$ is connected, we say that $\lambda \in \mathcal{P}^{(k)}$ is connected. Otherwise, $\lambda$ can be decomposed as $\lambda = \lambda_1 \sqcup \cdots \sqcup \lambda_l$. In this case, we define $p_\lambda = p_\lambda^{(k)}$ by:
\[
p_\lambda := m_{\lambda_1} \cdots m_{\lambda_l}.
\]
This corresponds to the classical case of power sum symmetric functions when $k=1$.

\begin{df}[\cite{MMNST}]
Let $G$ be a connected graph with $n$ vertices. We say that $\lambda \in \mathcal{P}_n^{(k)}$ is admissible by $G$ if there exists a bijection $\varphi: V(G) \to E_\lambda$ such that $\{u, v\} \in E(G) \iff \varphi(u) \cap \varphi(v) \neq \emptyset$. The set of elements in $\mathcal{P}_n^{(k)}$ that are admissible by $G$ is denoted by $\mathcal{A}_G^{(k)}$.

If $G$ is disconnected and $G = G_1 \sqcup \cdots \sqcup G_l$, then $\mathcal{A}_G^{(k)} := \mathcal{A}_{G_1}^{(k)} \times \cdots \times \mathcal{A}_{G_l}^{(k)}$.
\end{df}

With this, we obtain the following generalization of the power sum expansion.

\begin{thm}[Restatement of Theorem \ref{k-multiple_powersum}, \cite{MMNST}]
\[
X_{K_{\mathbb{N},k}}(G) = \sum_{S \subseteq E(G)} (-1)^{|S|} \sum_{\lambda \in \mathcal{A}_{G_S}^{(k)}} p_\lambda.
\]
\end{thm}

Now, we extend this to vertex-weighted graphs.

\begin{df}
Let $(G, w)$ be a connected vertex-weighted graph with $n$ vertices, and let the total weight be $N=\sum_{v\in V(G)}w(v)$. We say that $\lambda \in \mathcal{P}_N^{(k)}$ is admissible by $(G, w)$ if there exists a bijection $\varphi: V(G) \to E_\lambda$ such that $\{u, v\} \in E(G) \iff \varphi(u) \cap \varphi(v) \neq \emptyset$ and the multiplicity of $\varphi(v)\in E_\lambda$ equals $w(v)$. The set of elements in $\mathcal{P}_N^{(k)}$ that are admissible by $(G, w)$ is denoted by $\mathcal{A}_{(G, w)}^{(k)}$.

If $G$ is disconnected and the vertex-weighted graph $(G, w)$ decomposes as $(G, w) = (G_1, w_1) \sqcup \cdots \sqcup (G_l, w_l)$, then 
\[
\mathcal{A}_{(G, w)}^{(k)} = \mathcal{A}_{(G_1, w_1)}^{(k)} \times \cdots \times \mathcal{A}_{(G_l, w_l)}^{(k)}.
\]
\end{df}

\begin{lem}\label{lem:weak-p_lambda}
For $\lambda = (\lambda_1, \cdots, \lambda_l)$, define $p_\lambda = p_{\lambda_1} \cdots p_{\lambda_l}$. Then,
\[
W_{\overline{K_{\mathbb{N}, k}}}(G, w) = \sum_{\lambda \in \mathcal{A}_{(G, w)}^{(k)}} p_\lambda.
\]
\end{lem}
\begin{proof}
This result follows by replacing $G$ with $(G,w)$ in the proof of Lemma 4.4 in \cite{MMNST}.
If $G$ is connected, we have
\[
W_{\overline{K_{\mathbb{N}, k}}}(G, w) = \sum_{\varphi \in \wHom(G, \overline{K_{\mathbb{N}, k}})} \prod_{v \in V(G)} x_{\varphi(v)}^{w(v)} = \sum_{\lambda \in \mathcal{A}_{(G, w)}^{(k)}} m_\lambda = \sum_{\lambda \in \mathcal{A}_{(G, w)}^{(k)}} p_\lambda.
\]
If $G$ is disconnected and can be decomposed into its connected components $G = G_1 \sqcup \cdots \sqcup G_l$, then, via the natural bijection
\[
\wHom(G, \overline{K_{\mathbb{N}, k}}) \simeq \prod_{i=1}^l \wHom(G_i, \overline{K_{\mathbb{N}, k}}),
\]
we obtain
\[
W_{\overline{K_{\mathbb{N}, k}}}(G, w) = \prod_{i=1}^l W_{\overline{K_{\mathbb{N}, k}}}(G_i, w) = \prod_{i=1}^l \sum_{\lambda_i \in \mathcal{A}_{(G_i, w)}^{(k)}} p_\lambda = \sum_{\lambda \in \mathcal{A}_{(G, w)}^{(k)}} p_\lambda.
\]
\end{proof}

\begin{thm}
\[
X_{K_{\mathbb{N}, k}}(G, w) = \sum_{S \subseteq E} (-1)^{|S|} \sum_{\lambda \in \mathcal{A}_{(G_S, w)}^{(k)}} p_\lambda.
\]
\end{thm}
\begin{proof}
    It follows from Proposition \ref{prop:weak_expansion} and Lemma \ref{lem:weak-p_lambda}.
\end{proof}

\begin{ex}
Consider the  case $k=2$ and $(G,w) = (P_{3},w)$ as shown in Figure 1. Define the weight function $w$ by $w(v_1)=2$, $w(v_2)=w(v_3)=1$.

\begin{figure}[htbp]
\centering

\begin{tikzpicture}[scale=0.4]

\coordinate[label=above:$v_1$, label=below:$2$](v1)at(0,0);
\coordinate[label=above:$v_2$, label=below:$1$](v2)at(5,0);
\coordinate[label=above:$v_3$, label=below:$1$](v3)at(10,0);

\draw[thick](v1)--(v2);
\draw[thick](v2)--(v3);
\foreach\P in{v1,v2,v3}\fill[black](\P)circle(0.25);

\end{tikzpicture}

\caption{}
\end{figure}

Then the spanning subgraphs of $G$ are isomorphic to one of 
\begin{align*}
&S_{1|2|3}:=\begin{tikzpicture}[scale=0.15]
\coordinate[label=above:$v_1$](v1)at(0,0);
\coordinate[label=above:$v_2$](v2)at(5,0);
\coordinate[label=above:$v_3$](v3)at(10,0);
\foreach\P in{v1,v2,v3}\fill[black](\P)circle(0.5);
\end{tikzpicture},\\
&S_{12|3}:=\begin{tikzpicture}[scale=0.15]
\coordinate[label=above:$v_1$](v1)at(0,0);
\coordinate[label=above:$v_2$](v2)at(5,0);
\coordinate[label=above:$v_3$](v3)at(10,0);
\draw[thick](v1)--(v2);
\foreach\P in{v1,v2,v3}\fill[black](\P)circle(0.5);
\end{tikzpicture},\\
&S_{1|23}:=\begin{tikzpicture}[scale=0.15]
\coordinate[label=above:$v_1$](v1)at(0,0);
\coordinate[label=above:$v_2$](v2)at(5,0);
\coordinate[label=above:$v_3$](v3)at(10,0);
\draw[thick](v2)--(v3);
\foreach\P in{v1,v2,v3}\fill[black](\P)circle(0.5);
\end{tikzpicture},\\
&G=\begin{tikzpicture}[scale=0.15]
\coordinate[label=above:$v_1$](v1)at(0,0);
\coordinate[label=above:$v_2$](v2)at(5,0);
\coordinate[label=above:$v_3$](v3)at(10,0);
\draw[thick](v1)--(v2);
\draw[thick](v2)--(v3);
\foreach\P in{v1,v2,v3}\fill[black](\P)circle(0.5);
\end{tikzpicture}.
\end{align*}

Define

\begin{align*}
&G_1:=\begin{tikzpicture}[scale=0.15]
\coordinate[label=above:$v_1$](v1)at(0,0);
\foreach\P in{v1}\fill[black](\P)circle(0.5);
\end{tikzpicture},\quad
G_2:=\begin{tikzpicture}[scale=0.15]
\coordinate[label=above:$v_2$](v2)at(0,0);
\foreach\P in{v2}\fill[black](\P)circle(0.5);
\end{tikzpicture},\quad
G_3:=\begin{tikzpicture}[scale=0.15]
\coordinate[label=above:$v_3$](v3)at(0,0);
\foreach\P in{v3}\fill[black](\P)circle(0.5);
\end{tikzpicture},\\
&G_{12}:=\begin{tikzpicture}[scale=0.15]
\coordinate[label=above:$v_1$](v1)at(0,0);
\coordinate[label=above:$v_2$](v2)at(5,0);
\draw[thick](v1)--(v2);
\foreach\P in{v1,v2}\fill[black](\P)circle(0.5);
\end{tikzpicture},\quad
G_{23}:=\begin{tikzpicture}[scale=0.15]
\coordinate[label=above:$v_2$](v2)at(5,0);
\coordinate[label=above:$v_3$](v3)at(10,0);
\draw[thick](v2)--(v3);
\foreach\P in{v2,v3}\fill[black](\P)circle(0.5);
\end{tikzpicture}.
\end{align*}

For each $i, j \in \{1, 2, 3\}$, let $w_i$ denote the restriction of $w$ to $\{v_i\}$, and $w_{ij}$ denote the restriction of $w$ to $\{v_i, v_j\}$, where $i \ne j$.

Then we obtain
\begin{align*}
\mathcal{A}_{(S_{1|2|3},w)}^{(2)} 
&= \mathcal{A}_{(G_{1},w_1)}^{(2)} \times \mathcal{A}_{(G_{2},w_2)}^{(2)} \times \mathcal{A}_{(G_{3},w_3)}^{(2)} 
= \Set{\left(\begin{tikzpicture}[baseline=2]
\draw (0,0) node[v](1){};
\draw (0,0.4) node[v](2){};
\draw (1) to [bend left] (2);
\draw (1) to [bend right] (2);
\end{tikzpicture} \ \begin{tikzpicture}[baseline=2]
\draw (0,0) node[v](1){};
\draw (0,0.4) node[v](2){};
\draw (1)--(2);
\end{tikzpicture} \ \begin{tikzpicture}[baseline=2]
\draw (0,0) node[v](1){};
\draw (0,0.4) node[v](2){};
\draw (1)--(2);
\end{tikzpicture} \right)},  \\
\mathcal{A}_{(S_{12|3},w)}^{(2)} 
&= \mathcal{A}_{(G_{12},w_{12})}^{(2)} \times \mathcal{A}_{(G_{3},w_3)}^{(2)}
= \Set{\left( \begin{tikzpicture}[baseline=2]
\draw (0,-0.1) node[v](1){};
\draw (0,0.2) node[v](2){};
\draw (0,0.5) node[v](3){};
\draw (1)--(2);
\draw (2) to [bend left] (3);
\draw (2) to [bend right] (3);
\end{tikzpicture} \ \begin{tikzpicture}[baseline=2]
\draw (0,0) node[v](1){};
\draw (0,0.4) node[v](2){};
\draw (1)--(2);
\end{tikzpicture} \right), \left( \begin{tikzpicture}[baseline=2]
\draw (0,-0.1) node[v](1){};
\draw (0,0.5) node[v](2){};
\draw (1)--(2);
\draw (1) to [bend left] (2);
\draw (1) to [bend right] (2);
\end{tikzpicture} \ \begin{tikzpicture}[baseline=2]
\draw (0,0) node[v](1){};
\draw (0,0.4) node[v](2){};
\draw (1)--(2);
\end{tikzpicture} \right)}, \\
\mathcal{A}_{(S_{1|23},w)}^{(2)} 
&= \mathcal{A}_{(G_{1},w_{1})}^{(2)} \times \mathcal{A}_{(G_{23},w_{23})}^{(2)}
= \Set{\left( \begin{tikzpicture}[baseline=2]
\draw (0,0) node[v](1){};
\draw (0,0.4) node[v](2){};
\draw (1) to [bend left] (2);
\draw (1) to [bend right] (2);
\end{tikzpicture} \ \begin{tikzpicture}[baseline=2]
\draw (0,-0.1) node[v](1){};
\draw (0,0.2) node[v](2){};
\draw (0,0.5) node[v](3){};
\draw (1)--(2)--(3);
\end{tikzpicture} \right), \left( \begin{tikzpicture}[baseline=2]
\draw (0,0) node[v](1){};
\draw (0,0.4) node[v](2){};
\draw (1) to [bend left] (2);
\draw (1) to [bend right] (2);
\end{tikzpicture} \ \begin{tikzpicture}[baseline=2]
\draw (0,0) node[v](1){};
\draw (0,0.4) node[v](2){};
\draw (1) to [bend left] (2);
\draw (1) to [bend right] (2);
\end{tikzpicture} \right)}, \\
\mathcal{A}_{(G,w)}^{(2)} 
&= \Set{\begin{tikzpicture}[baseline=3]
\draw (0,-0.2) node[v](1){};
\draw (0,0.1) node[v](2){};
\draw (0,0.4) node[v](3){};
\draw (0,0.7) node[v](4){};
\draw (3) to [bend left] (4);
\draw (3) to [bend right] (4);
\draw (1)--(2)--(3);
\end{tikzpicture}, \ \begin{tikzpicture}[baseline=-5]
\draw (0,-0.1) node[v](1){};
\draw (0,0.3) node[v](2){};
\draw (-0.2,-0.4) node[v](3){};
\draw ( 0.2,-0.4) node[v](4){};
\draw (1) to [bend left] (2);
\draw (1) to [bend right] (2);
\draw (1)--(3);
\draw (1)--(4);
\end{tikzpicture}, \ \begin{tikzpicture}
\draw (0,-0.3) node[v](1){};
\draw (0.4,-0.3) node[v](2){};
\draw (0.2,0) node[v](3){};
\draw (2)--(3)--(1);
\draw (1) to [bend left] (2);
\draw (1) to [bend right] (2);
\end{tikzpicture}, \ \begin{tikzpicture}[baseline=2]
\draw (0,-0.2) node[v](1){};
\draw (0,0.2) node[v](2){};
\draw (0,0.6) node[v](3){};
\draw (1)--(2);
\draw (2)--(3);
\draw (2) to [bend left] (3);
\draw (2) to [bend right] (3);
\end{tikzpicture}, \ \begin{tikzpicture}[baseline=2]
\draw (0,-0.2) node[v](1){};
\draw (0,0.2) node[v](2){};
\draw (0,0.6) node[v](3){};
\draw (1) to [bend left] (2);
\draw (1) to [bend right] (2);
\draw (2) to [bend left] (3);
\draw (2) to [bend right] (3);
\end{tikzpicture}, \ \begin{tikzpicture}[baseline=2]
\draw (0,0) node[v](1){};
\draw (0,0.6) node[v](2){};
\draw (1) to [bend left = 10] (2);
\draw (1) to [bend right = 10] (2);
\draw (1) to [bend left = 30] (2);
\draw (1) to [bend right = 30] (2);
\end{tikzpicture}  }. 
\end{align*}

Therefore 

\begin{align*}
&X_{K_{\mathbb{N},2}}(G,w)\\
&= \sum_{\lambda \in \mathcal{A}_{(S_{1|2|3},w)}^{(2)}}p_{\lambda} -\sum_{\lambda \in \mathcal{A}_{(S_{12|3},w)}^{(2)}}p_{\lambda} - \sum_{\lambda \in \mathcal{A}_{(S_{1|23},w)}^{(2)}}p_{\lambda} + \sum_{\lambda \in \mathcal{A}_{(G,w)}^{(2)}}p_{\lambda} \\
&= p_{\, \begin{tikzpicture}
\draw (0,0) node[v](1){};
\draw (0,0.4) node[v](2){};
\draw (1) to [bend left] (2);
\draw (1) to [bend right] (2);
\end{tikzpicture} \, \begin{tikzpicture}
\draw (0,0) node[v](1){};
\draw (0,0.4) node[v](2){};
\draw (1)--(2);
\end{tikzpicture} \, \begin{tikzpicture}
\draw (0,0) node[v](1){};
\draw (0,0.4) node[v](2){};
\draw (1)--(2);
\end{tikzpicture}}-p_{\, \begin{tikzpicture}[baseline=2]
\draw (0,-0.1) node[v](1){};
\draw (0,0.2) node[v](2){};
\draw (0,0.5) node[v](3){};
\draw (1)--(2);
\draw (2) to [bend left] (3);
\draw (2) to [bend right] (3);
\end{tikzpicture} \, \begin{tikzpicture}[baseline=2]
\draw (0,0) node[v](1){};
\draw (0,0.4) node[v](2){};
\draw (1)--(2);
\end{tikzpicture}}-p_{\, \begin{tikzpicture}[baseline=2]
\draw (0,-0) node[v](1){};
\draw (0,0.4) node[v](2){};
\draw (1)--(2);
\draw (1) to [bend left] (2);
\draw (1) to [bend right] (2);
\end{tikzpicture} \, \begin{tikzpicture}[baseline=2]
\draw (0,0) node[v](1){};
\draw (0,0.4) node[v](2){};
\draw (1)--(2);
\end{tikzpicture}} - p_{\, \begin{tikzpicture}[baseline=2]
\draw (0,-0) node[v](1){};
\draw (0,0.4) node[v](2){};
\draw (1) to [bend left] (2);
\draw (1) to [bend right] (2);
\end{tikzpicture} \, \begin{tikzpicture}[baseline=2]
\draw (0,-0.1) node[v](1){};
\draw (0,0.2) node[v](2){};
\draw (0,0.5) node[v](3){};
\draw (1)--(2);
\draw (2)--(3);
\end{tikzpicture}} - p_{\, \begin{tikzpicture}[baseline=2]
\draw (0,-0) node[v](1){};
\draw (0,0.4) node[v](2){};
\draw (1) to [bend left] (2);
\draw (1) to [bend right] (2);
\end{tikzpicture} \, \begin{tikzpicture}[baseline=2]
\draw (0,-0) node[v](1){};
\draw (0,0.4) node[v](2){};
\draw (1) to [bend left] (2);
\draw (1) to [bend right] (2);
\end{tikzpicture}} + p_{\, \begin{tikzpicture}[baseline=3]
\draw (0,-0.1) node[v](1){};
\draw (0,0.2) node[v](2){};
\draw (0,0.5) node[v](3){};
\draw (0,0.8) node[v](4){};
\draw (1)--(2)--(3);
\draw (3) to [bend left] (4);
\draw (3) to [bend right] (4);
\end{tikzpicture}} + p_{ \begin{tikzpicture}[baseline=-5]
\draw (0,0) node[v](1){};
\draw (0,0.3) node[v](2){};
\draw (-0.15,-0.26) node[v](3){};
\draw ( 0.15,-0.26) node[v](4){};
\draw (1) to [bend left] (2);
\draw (1) to [bend right] (2);
\draw (1)--(3);
\draw (1)--(4);
\end{tikzpicture}} + p_{\begin{tikzpicture}
\draw (0,0) node[v](1){};
\draw (0.3,0) node[v](2){};
\draw (0.15,0.26) node[v](3){};
\draw (2)--(3)--(1);
\draw (1) to [bend left] (2);
\draw (1) to [bend right] (2);
\end{tikzpicture}} + p_{\, \begin{tikzpicture}[baseline=3]
\draw (0,-0.2) node[v](1){};
\draw (0,0.2) node[v](2){};
\draw (0,0.6) node[v](3){};
\draw (1)--(2);
\draw (2)--(3);
\draw (2) to [bend left] (3);
\draw (2) to [bend right] (3);
\end{tikzpicture}} + p_{\, \begin{tikzpicture}[baseline=3]
\draw (0,-0.1) node[v](1){};
\draw (0,0.2) node[v](2){};
\draw (0,0.5) node[v](3){};
\draw (1) to [bend left] (2);
\draw (1) to [bend right] (2);
\draw (2) to [bend left] (3);
\draw (2) to [bend right] (3);
\end{tikzpicture}} + p_{\begin{tikzpicture}[baseline=2]
\draw (0,0) node[v](1){};
\draw (0,0.6) node[v](2){};
\draw (1) to [bend left = 10] (2);
\draw (1) to [bend right = 10] (2);
\draw (1) to [bend left = 30] (2);
\draw (1) to [bend right = 30] (2);
\end{tikzpicture}}. 
\end{align*}
\end{ex}

\section{Construction of complete invariants for DAGs and posets}

Let $\vec{G}$ be a finite directed acyclic graph (DAG). Let $G$ be the undirected graph obtained by ignoring the directions in $\vec{G}$. Define a weight function $w$ on $G$ as follows:
\[
w(v) := \begin{cases}
1, & \text{if the in-degree of } v \text{ is } 0,\\
1 + \underset{u: u \to v}{\max} w(u), & \text{otherwise}.
\end{cases}
\]
Then, the function 
\[
Y_H(\vec{G}, x) := X_H((G, w), x)
\]
is uniquely determined, and the following two statements hold.

\begin{thm}
(1) Let $H$ be a universal graph. Then,
\[
Y_H(\bullet)
\]
is a complete invariant for finite DAGs.

(2) Let $H = \{H_n\}_{n\in\mathbb{N}}$ be an infinite sequence of universal graphs. Then,
\[
\{Y_{H_n}(\bullet)\}_{n \in \mathbb{N}}
\]
is a complete invariant for finite DAGs.
\end{thm}

\begin{proof}
We demonstrate that the weight assignment is both well-defined and injective. Let $\vec{G}_1$ and $\vec{G}_2$ be DAGs that are isomorphic as directed graphs. Let $\varphi: \vec{G}_1 \to \vec{G}_2$ be the isomorphism. Denote the weights assigned to the undirected versions of $\vec{G}_1$ and $\vec{G}_2$, denoted by $G_1$ and $G_2$ respectively, as $w_1$ and $w_2$. Assume there exists a vertex $v_0 \in V(\vec{G}_1)$ such that $w_1(v_0) \neq w_2(\varphi(v_0))$. Then, it must hold that $\underset{v_1: v_1 \to v_0}{\max} w_1(v_1) \neq \underset{v_1': v_1' \to \varphi(v_0)}{\max} w_2(v_1')$, and for at least one vertex $v_1$ (where there is an edge $v_1 \to v_0$), $w_1(v_1) \neq w_2(\varphi(v_1))$ holds. Repeating this process, we eventually arrive at a vertex with in-degree 0 that does not have a weight of 1, leading to a contradiction. Next, given any vertex-weighted graph $(G, w)$ obtained from a weight assignment on a DAG, the original DAG can be reconstructed as follows: The vertex set remains $V(G)$, and the arc set $A$ is defined by
\[
A := \{\{u, v\} \in \binom{V(G)}{2} \mid w(u) + 1 = w(v)\}.
\]
Thus, the weight assignments uniquely determine the arcs, ensuring that non-isomorphic DAGs have distinct weight assignments. Therefore, from Theorem \ref{main_thm}, $Y_H(\bullet)$ is a complete invariant.

\end{proof}

Let $P$ be a poset and let $\vec{G}_P$ denote the directed graph representing the Hasse diagram of $P$. Let $G_P$ be the undirected graph obtained by ignoring the directions in $\vec{G}_P$. Assign a weight $w_P$ to $G_P$ as described in the previous theorem. Then,
\[
Z_H(P, x) := Y_H(\vec{G}_P, x) = X_H((G_P, w_P), x)
\]
is uniquely determined by the structure of the poset and the weight assignment as described, and the following two statements hold.

\begin{cor}
(1) Let $H$ be a universal graph. Then,
\[
Z_H(\bullet)
\]
is a complete invariant for finite posets.

(2) Let $H = \{H_n\}_{n\in\mathbb{N}}$ be an infinite sequence of universal graphs. Then,
\[
\{Z_{H_n}(\bullet)\}_{n \in \mathbb{N}}
\]
is a complete invariant for finite posets.
\end{cor}

\begin{rem}
    $Z_H(\bullet)$ can distinguish hyperplane arrangements that are distinguishable by their intersection posets.
\end{rem}

\begin{prob}
(1) Can we define elementary $k$-fold symmetric functions and Schur $k$-fold symmetric functions? If so, when is $X_{K_{\NN,k}}(G,w)$ $e$-positive or $s$-positive?\\
(2) If $H$ is not a complete graph, the deletion-contraction relation: 
\[
X_H(G,w) = X_H(G\backslash e,w) - X_H(G/e,w/e)
\]  
does not generally hold. For instance, can a similar identity be derived by introducing a correction term?
\end{prob}

\section*{Acknowledgments}
The author thanks Professor Tsuyoshi Miezaki for his insightful lectures on graph theory, including discussions on the chromatic symmetric function. The author also expresses gratitude to Master's students Naoki Fujii, Yusaku Nishimura, and Ryosuke Yamaguchi for their valuable discussions.

\section*{Data Availability}
Data sharing is not applicable to this paper as no data sets were generated or analyzed during the current study.

\section*{Declarations}
\subsection*{Conflict of interest}
The author declares that they have no conflict of interest.


\end{document}